\documentclass[twoside,12pt]{amsart}
\usepackage{amsmath,latexsym,amssymb,mathptm, times,verbatim, enumerate,mathcomp,lscape}
    \usepackage{longtable}\usepackage{tikz,stmaryrd}
\usepackage{pdfpages}
\usepackage{calligra}
\usepackage{calrsfs}
\usepackage[mathscr]{euscript}
\input amssym.def
\input amssym 
\input xypic
\input xy
\xyoption{all}
\usepackage[right=1.5in,left=1.5in,top=1in,bottom=1in]{geometry}

\newcommand{\gf}[2]{\genfrac{}{}{0pt}{}{#1}{#2}}

\newcommand{\qci}{{q.c.i.\! }}

\newcommand{\cls}{\operatorname{cls}}
\def\Fdagger{\mathfrak F}
\def\Gdagger{\mathfrak G}
\def\Pdagger{\mathfrak P}
\def\Tdagger{\mathfrak T}
\def\Ftil{\widetilde{F}}
\def\Gtil{\widetilde{G}}
\def\mcI{\mathcal I}
\def\xu{\underline{x}}
\def\Ring{\mathfrak R}
\def\ov{\overline}
\def \Rbar{\ov{\Ring}}
\def\Fbar{\ov{\Fdagger}}
\def\Gbar{\ov{\Gdagger}}
\def\ff{\goth f}
\def\gog{\goth g}

\def\M{\mathfrak M}

\def\ev{\operatorname{ev}}
\def\la{\langle}
\def\ra{\rangle}
\def\phibar{\ov{\Phi}}
\def\Xibar{\ov{\Xi}}

\def\Rtil{\widetilde{R}}
\def\done{\xi}
\def\blop{\phi}

\def\depth{\operatorname{depth}}
\def\CIdim{\operatorname{CI-dim}}
\def\Gdim{\operatorname{G-dim}}
\def\Max{\operatorname{Max}}
\def\pd{\operatorname{pd}}
\def\cx{\operatorname{cx}}

\def\p{\oplus}
\def\im{\operatorname{im}}
\def\mult{\operatorname{mult}}
\def\m{\mathfrak m}

\def\a{\alpha}

\def\Sym{\operatorname{Sym}}
\def\Hom{\operatorname{Hom}}
\def\kk {\pmb k}
\def\Tor{\operatorname{Tor}}
\def\t{\otimes}
\def\w{\wedge}
\def\HH{\operatorname{H}}

\def\grade{\operatorname{grade}}

\def\Ext{\operatorname{Ext}}
\def\ts{\textstyle}

\newtheorem{theorem}{Theorem}[section]

\newtheorem{corollary}[theorem]{Corollary}
\newtheorem{proposition}[theorem]{Proposition}
\newtheorem{proposition-no-advance}[equation]{Proposition}

\newtheorem{claim-no-advance}[equation]{Claim}
\newtheorem{observation}[theorem]{Observation}
\newtheorem{subtheorem}{Theorem}[theorem]

\newtheorem{question}[theorem]{Question}

\newtheorem{quick consequences}[theorem]{Quick Consequences}

\theoremstyle{definition}

\newtheorem{definition-no-advance}[equation]{Definition}
\newtheorem{construction}[theorem]{Construction}
\newtheorem{facts and definitions}[theorem]{Facts and Definitions}
\newtheorem{definition}[theorem]{Definition}
\newtheorem{remark}[theorem]{Remark}
\newtheorem{remark-no-advance}[equation]{Remark}
\newtheorem{remarks-no-advance}[equation]{Remarks}
\newtheorem{data}[theorem]{Data}

\newtheorem{careful calculation}[theorem]{Careful Calculation}

\newtheorem{present summary}[theorem]{Present Summary}

\newtheorem{examples}[theorem]{Examples}

\newtheorem{further reductions}[theorem]{Further Reductions}

\newtheorem{chunk}[theorem]{}
\newtheorem{chunk-no-advance}[equation]{}
\newtheorem{subchunk}[subtheorem]{}

\newtheorem{marching orders}[theorem]{Marching Orders}

\newtheorem{circle the wagons}[theorem]{Circle the wagons}

\newtheorem*{Remark}{Remark}

\numberwithin{equation}{theorem}
\numberwithin{table}{theorem}

\begin{document}

\baselineskip=16pt

\title [The structure of quasi-complete intersection ideals]{The structure of quasi-complete intersection ideals}

\date{\today}

\author[A.~R.~Kustin]{Andrew R.~Kustin}
\address{Andrew R.~Kustin\\ Department of Mathematics\\ University of South Carolina\\\newline 
Columbia\\ SC 29208\\ U.S.A.} \email{kustin@math.sc.edu}

\author[L.~M.~\c{S}ega]{Liana M.~\c{S}ega}
\address{Liana M.~\c{S}ega\\ Department of Mathematics and Statistics\\
   University of Missouri\\  Kansas City\\ MO 64110\\ U.S.A.}
     \email{segal@umkc.edu}



\thanks{Research partly supported by  Simons Foundation collaboration grant 233597 (ARK) and  Simons Foundation collaboration grant 354594 (LM\c{S})}

\begin{abstract}   We prove that every quasi-complete intersection (\qci\!\!)  ideal  is obtained from a pair of nested complete intersection ideals by way of a flat base change. As a by-product we establish a rigidity statement for the minimal two-step Tate complex associated to an ideal $I$ in a local ring $R$. 
Furthermore, we define a minimal two-step complete Tate complex $T$ for each ideal $I$ in a local ring $R$; and prove a rigidity result for it. The complex $T$ is exact if and only if $I$ is a \qci ideal; and in this case, $T$ is the minimal complete resolution of $R/I$ by free $R$-modules. 
\end{abstract}

\maketitle

\section{Introduction.}
Let $R\to S$ be a homomorphism of commutative Noetherian rings. Quillen \cite[5.6]{Qu70}  conjectured that if the Andr\'e-Quillen homology functors $D_i(S|R,-)$ vanish for all large $i$, then they vanish for $3\le i$. We investigate the structure of ideals $I$ in a local Noetherian ring $R$ for which the natural quotient map $R\to S=R/I$ satisfies the conclusion of the Quillen conjecture. Such ideals are called quasi-intersection (\qci\!\!) ideals; see \cite{AHS13}. (Other equivalent definitions are given in Section~\ref{2.B}.)
 The title of the present paper refers to Corollary~\ref{every} which states 
that  
every \qci ideal in a local Noetherian ring is obtained from a pair of nested complete intersection ideals by way of a ``flat base change'', in the sense of \cite[8.7]{AHS13}.

Given a particular \qci ideal $I$, the proof of Corollary~\ref{every} for $I$ involves creating a  generic pair of nested complete intersection ideals for $I$ in $R$. As such, we view the two-step Tate complex for $I$ as being obtained from the generic case by way of a base change. The generic case is the prototype of a module of finite complete intersection dimension ($\CIdim$). Theorems due to Jorgensen \cite{Jo99} and Avramov and Buchweitz \cite{AB00} prove that if $M$ is an $R$-module of finite  $\CIdim$, then $\Tor_i^R(M,N)$ can not vanish for many consecutive values of $i$ unless it vanishes for all large values of $i$. We use these Theorems to prove a strong rigidity result, Theorem~\ref{main}, about the two-step
Tate complex associated to an ideal. Corollary~\ref{Rigidity} recovers and extends a rigidity result of Jason Lutz \cite{L16} using apparently different methods. The rigidity results of Theorem~\ref{main} and Corollary~\ref{Rigidity} are prettiest when they are expressed in terms of Tate homology rather than  ordinary homology. As a consequence, we give an \underline{explicit} form for the complete resolution of each \qci ideal. 

 It was not possible to explain the
 \qci ideal $I$ of the ring $R$ in \cite[Sect.~4]{KSV14} using any of the techniques that appeared in \cite{KSV14}. Indeed, this example seemed to indicate that \qci ideals could be arbitrarily complicated. However, we show in Example~\ref{Ex}.(\ref{Ex.d}) that 
$I$ is obtained from a pair of nested complete intersection ideals $A\subseteq B$ in a local ring $\Ring$ by way of a base change $\Ring/A\to R$ and that $R$ has finite projective dimension as a module over $\Ring/A$. 
We 
apply the rigidity portion of Theorem~\ref{main} in order to conclude that the base change 
 is flat; hence, the unexplained example in \cite{KSV14} is accounted for by Corollary~\ref{every}. 

\tableofcontents

\section{Notation, conventions, and preliminary results.}\label{Prelims}

\subsection{Terminology}
\begin{chunk} When it is clear that ``$R$'' is the ambient ring, we use $(-)^*$ to mean $\Hom_R(-,R)$. 

\begin{subchunk}``Let $(R,\mathfrak m,\kk)$ be a local ring'' identifies $\mathfrak m$ as the unique maximal ideal of the commutative Noetherian local ring $R$ and $\kk$ as the residue class field $\kk=R/\mathfrak m$. If $M$ is a finitely generated $R$-module, we denote by $\mu(M)$ the minimal number of generators of $M$.
\end{subchunk}

\begin{subchunk}We use 
``$\im$'', ``$\ker$'', ``$\mult$'', and ``$\pd$''
 as  abbreviations for ``image'', ``kernel'', ``multiplication'', and ``projective dimension'',  respectively.\end{subchunk} 

\begin{subchunk}If $M$ is a finitely generated module over a local ring $(R,\m)$ and $\phi:M\to N$ is an $R$-module homomorphism, then $\phi$ is a {\it minimal homomorphism} if $\ker\phi\subseteq \m M$.
\end{subchunk} 

\begin{subchunk}The {\it grade} of an ideal $I$  in a commutative Noetherian ring $R$ is the length of a maximal regular sequence on $R$ which is 
contained in $I$.\end{subchunk}

\begin{subchunk}If $\mathbb Y$ is a complex, then we use $Z_i(\mathbb Y)$,  $B_i(\mathbb Y)$, and $\HH_i(\mathbb Y)$ to represent the modules of $i$-cycles, $i$-boundaries, and $i^{\text{th}}$-homology  of $\mathbb Y$, respectively. If $z$ is an $i$-cycle, then $\cls(z)$ is the homology class of $z$ in $\HH_i$. In a similar manner, $Z^i$, $B^i$, and $\HH^i$ represent co-cycles, co-boundaries, and cohomology, respectively. A {\it quasi-isomorphism} is a homomorphism of  complexes that induces an isomorphism on homology. \end{subchunk} 

\begin{subchunk}If $\Phi$ is a matrix (or a homomorphism of finitely generated free $R$-modules), then $I_r(\Phi)$ is the ideal generated by the
$r\times r$ minors of $\Phi$ (or any matrix representation of $\Phi$). \end{subchunk}

\begin{subchunk}Let $F$ be a free $R$-module of finite rank, $F^*=\Hom_R(F,R)$,  $\ev:F\t F^*\to R$ be the evaluation map, and $\ev^*:R\to F^*\t F$ be the dual of the evaluation map. If $(\{x_i\},
\{x_i^*\})$ is a pair of dual bases for $F$ and $F^*$, respectively, then  $$\ts \ev^*(1)=\sum_i x_i^*\t x_i.$$ Of course, $\sum_i x_i^*\t x_i$ is a canonical element of $F^*\t F$. \end{subchunk} 

\begin{subchunk}If $F$ is a free $R$-module of finite rank $\ff$, then the exterior algebra
$\bigwedge^{\bullet}F$ is a module over the graded-commutative ring $\bigwedge^{\bullet}F^*$. In particular, if $\xi\in F^*$, then $(\bigwedge^\bullet F,\xi)$ is the Koszul complex 
$$\ts 0\to \bigwedge^\ff F\xrightarrow{\xi}
\bigwedge^{\ff-1} F\xrightarrow{\xi}\cdots \xrightarrow{\xi}
\bigwedge^{2} F\xrightarrow{\xi} F\xrightarrow{\xi}R.$$ We refer to $(\bigwedge^\bullet F,\xi)$ as {\it the Koszul complex associated to $\xi$}.
In a similar manner, the divided power algebra $D_\bullet F^*$ is a module over the polynomial ring $\Sym_\bullet F$.
\end{subchunk}

\begin{subchunk}\label{2.1.9}If $I$ is an ideal in a local ring $R$ and $\xi:F\to R$ is a minimal homomorphism with $\im \xi=I$, then $\mu(\HH_1(\bigwedge^\bullet F,\xi))$ does not depend on $\xi$ (\cite[1.6.21]{BH}) and we denote this number by 
$\mu(K_1(I))$. We refer to this number as ``the minimal number of  generators of the first Koszul homology  module associated to a minimal generating set for $I$''.
\end{subchunk} 
\end{chunk}

\subsection{Quasi-complete intersections and the two-step Tate Complex.}\label{2.B}
\begin{data}\label{9.1}Let $I$ be an ideal in a local ring $(R,\m)$, $F$ be a free $R$-module of finite rank, $\done:F\to R$  be an $R$-module homomorphism with the image of $\done$ equal to $I$, and $\bigwedge^\bullet F$ be the Koszul complex associated to $\done$. \end{data}

\begin{definition-no-advance}\label{9.1.1}In the setup of {\rm\ref{9.1}}, the ideal $I$ is  a {\it quasi-complete intersection} (\qci\!)  if
\begin{enumerate}[\rm(a)]
\item\label{9.1.a} $\HH_1(\bigwedge^\bullet F)$ is a free $R/I$-module and
\item\label{9.1.b} the natural map 
$\bigwedge^\bullet \big(\HH_1(\bigwedge^\bullet F)\big) \to \HH_{\bullet} (\bigwedge^\bullet F)$ is an isomorphism of graded $R/I$-algebras. 
\end{enumerate}
\end{definition-no-advance}

\begin{remarks-no-advance}\begin{enumerate}[\rm(a)]\item The defining conditions \ref{9.1.1}.(\ref{9.1.a}) and \ref{9.1.1}.(\ref{9.1.b}) of \qci ideals do not depend on the choice of presentation for $R/I$ because the ambient ring $R$ is local; see, for example, \cite[1.6.21]{BH}.
\item The transition from \qci ideals as defined in the introduction to \qci ideals as defined in \ref{9.1.1} is contained in \cite[8.5]{AHS13}.
\item Ideals which satisfy properties \ref{9.1.1}.(\ref{9.1.a}) and \ref{9.1.1}.(\ref{9.1.b}) were first studied in \cite{R95}; they were named {\it ideals with free exterior homology} in \cite{BMR96}.
\end{enumerate}
\end{remarks-no-advance}

The two-step Tate complex detects \qci ideals. 
\begin{definition-no-advance}\label{two-step} Adopt the setup of \ref{9.1}. Let $G$ be a free $R$-module and $$\ts \blop:G\to Z_1(\bigwedge^\bullet F)$$ be an $R$-module homomorphism with the property that the composition 
$$\ts G\xrightarrow{\blop}Z_1(\bigwedge^\bullet F)\xrightarrow{\text{natural quotient map}}\HH_1(\bigwedge^\bullet F)$$ is a minimal surjection. Then the Divided Power Algebra $$\ts P=\la  \bigwedge^\bullet F\t_{R} D_{\bullet} {G},\partial\ra,$$ where the restriction of the differential $\partial$ to $F$ is given by $\done$ and the restriction of $\partial$ to $G$ is given by $\blop$, is called the {\it two-step Tate complex} associated to the data $(\done,\blop)$. 
\end{definition-no-advance}
 
\begin{remark-no-advance}\label{TateConstruction}A coordinate-dependent formulation of the two-step Tate complex may be found in \cite[1.5]{AHS13} and many other places. In this alternate language, $P$ is called a ``Tate construction'' and is written
$$R\la v_1,\ldots,v_{\ff}; w_1,\ldots,w_{\gog}\ra,$$where the exterior variables   $v_1,\ldots,v_{\ff}$ are a basis for $F$ and the divided power variables $w_1,\ldots,w_{\gog}$ are a basis for $G$.
\end{remark-no-advance}

\begin{proposition-no-advance}\label{2.2.4} {\rm (\cite[Thm.~1]{BMR98})}
In the language of Definition~{\rm\ref{two-step}},
the two-step Tate complex associated to the data $(\done,\blop)$ is acyclic if and only if $I$ is a \qci ideal.\end{proposition-no-advance}
 
\begin{Remark}  Observe that in Proposition~\ref{2.2.4} the homomorphism $\done$   need not be minimal, but the homomorphism $\blop$ must be minimal.  On the other hand, if $\done$ is a minimal homomorphism and the conditions of \ref{2.2.4} hold, then the two-step Tate complex associated to the data $(\done,\blop)$ is a minimal resolution of $R/I$ by free $R$-modules.\end{Remark}

\begin{proposition-no-advance}\label{2.2.5}Adopt the setup of Definition~{\rm\ref{9.1.1}}. Let $\xu=x_1,\dots,x_r$ be the beginning of a minimal generating set for $I$ which is also a regular sequence on $R$, $I'=I/(\xu)$, and $R'=R/(\xu)$.
\begin{enumerate}[\rm(a)]
\item\label{2.2.5.a} Natural data   $(\done',\blop')$ for $I'$ in $R'$ can be constructed from the given data $$\ts\big(\done:F\to R,\blop:G_1\to Z_1(\bigwedge^\bullet F)\big)$$ for $I$ in $R$.
\item\label{2.2.5.b} There is a quasi-isomorphism from the two-step Tate complex associated to the data $(\done,\blop)$ to
the two-step Tate complex associated to the data  $(\done',\blop')$. 
\end{enumerate} 
 \end{proposition-no-advance}

\begin{proof} (\ref{2.2.5.a})
Let $X$ be a free summand of $F$ with $\done(X)=(\xu)$. Define
$$F'=(F/X)\t_R R',$$ $\done':F'\to R'$ to be the map induced by 
$$F\t_RR'\xrightarrow {\done\t_R1} R\t_R R',$$
$G'=G\t_RR'$, and $\blop': G'\to F'$ to be the composition
$$G'=G\t_RR' \xrightarrow{\blop\t 1}F\t_RR'\xrightarrow{\text{natural quotient map}}F/X\t_R R'=F'.$$ The proof of \cite[Lem~1.3]{AHS13} shows that the natural map of Koszul complexes \begin{equation}\label{nmoKc}\ts(\bigwedge^\bullet F,\done)\to
(\bigwedge^\bullet F',\done')\end{equation} is a quasi-isomorphism. It follows, in particular, that the composition 
$$\ts G'\xrightarrow{\blop'} Z_1(\bigwedge^\bullet F') \xrightarrow{\text{natural quotient map}} \HH_1(\bigwedge^\bullet F')$$ is a minimal surjection.

\medskip\noindent(\ref{2.2.5.b})
The quasi-isomorphism (\ref{nmoKc}) can be extended to a quasi-isomorphism of the two-step Tate complexes by 
\cite[1.3.5]{GL}. \end{proof}

\subsection{Complete resolutions and Tate homology.}
\begin{chunk}Let $R$ be a commutative Noetherian ring, $(-)^*$ represent $\Hom_R(-,R)$, and 
$M$ be a finitely generated $R$-module. 
\begin{subchunk}\label{2.1.1}A {\it complete resolution} of  $M$ is a complex $T$ of finitely generated
projective $R$-modules, such that $\HH_i (T) = 0 = \HH^i (T^*)$ for all integers $i$, and
$T_{\ge r} =F_{\ge  r}$ for some projective resolution $F$ of M and some integer $r$. If $M$ has complete resolutions, then any two of them are
homotopy equivalent, see, for example  
\cite[Lem.~2.4]{CK97}.
In particular, in this case, the modules
$\widehat{\Tor}_i^ R (M, N ) = \HH_ i (T \t_R N )$
are well defined for all $R$-modules $N$; we refer to these modules as  
Tate $\Tor$ modules. 
\end{subchunk}

\begin{subchunk} The module $M$ is {\em totally reflexive} if 
$M\cong M^{**}$ 
 and $$\Ext^i_R (M, R)=\Ext^i_R (M^* , R)=0,$$
 for all positive $i$. \end{subchunk}

\begin{subchunk} If the module $M$ is not zero, then the $G$-{\it dimension} of $M$  is the length of the shortest resolution
of $M$ by totally reflexive $R$-modules.\end{subchunk}

\begin{subchunk}\label{2.1.4} The module $M$ has a complete resolution if and only if 
the $G$-dimension of $M$ is finite; see, for example, \cite[4.4.4]{AB00}.
\end{subchunk}

\begin{subchunk} If the ring $(R,\m)$ is local and the complete resolution 
$(T,d)$ of $M$ satisfies $d(T)\subseteq \m T$, then $T$ is a {\it minimal complete resolution} of $M$.  
It is shown in \cite [Thm.~8.4]{AM02} 
that  
any two minimal complete resolutions of M are isomorphic. 
\end{subchunk}
\end{chunk}

\subsection{Complete intersection dimension.}
\begin{chunk}
A {\it quasi-deformation} (of codimension $c$) of a local ring $R$ is a diagram
of local homomorphisms $R \to R' \leftarrow Q$, in which the left-most map is   faithfully flat and
the right-most map is surjective with kernel generated by a regular sequence on $Q$ (of
length $c$).
\begin{subchunk}\label{2.3.1}
Let $M$ be a non-zero finitely generated module over a Noetherian ring $R$. If $R$ is local, then
$$\CIdim_R M =\inf\{ \pd_Q(M\t_RR')-\pd_QR'\mid R\to R'\leftarrow Q \text{ is a quasi-deformation}\};$$
in general, the {\it complete intersection dimension} of $M$ over $R$ is
defined by
$$\CIdim_R M = \sup\{\CIdim R_{\m} M_{\m}\mid \m \in \Max(R)\}\quad\text{and}\quad
\CIdim_R 0 = 0,$$ 
where $\pd$ means ``projective dimension'' and $\Max$ means ``maximal spectrum''.
\end{subchunk}

\begin{subchunk}\label{2.2.2}It is shown in \cite[Thm.~1.4]{AGP97} that if $M$ is 
a finitely generated  module $M$ over a Noetherian ring $R$ then 
$$\Gdim_R M \le \CIdim_R M \le \pd_R M.$$
Furthermore, if any of these dimensions is finite, then this dimension  is equal to all dimensions  to its left.
Also, if $R$ is local and $\CIdim_R M < \infty$, then $$\CIdim_R M = \depth R - \depth_R M.$$
\end{subchunk}
\end{chunk}

\subsection{Complexity.}

\begin{chunk}\label{2.4} Let $M$ be a finitely generated module over the local ring $(R,\m,\kk)$. The {\it complexity} of $M$ is equal to 
$$\cx_R M = \inf\left\{ \text{non-negative integers $d$} \left\vert \begin{array}{l}\text{there exists a positive real number $\gamma$}\\\text{with $b_i^R(M)\le \gamma i^{d-1}$ for   $0\ll i$} \end{array}\right.\right\},$$ 
where $b_i^R(M)$ is the $i^{\text{th}}$-Betti number
$\dim_{\kk}\Tor_R^i(\kk,M)$
of the $R$-module $M$. 
If the CI-dimension of $M$ is finite, then \cite[Thm.~5.3]{AGP97} proves that the complexity of $M$ is finite and is equal to the order of the pole at $t=1$ of the Poincar\'e series $$P_M^R(t)=\sum_{i=0}^\infty b_i^R(M) t^i.$$
\end{chunk}

\subsection{The vanishing theorem for homology of modules of finite CI-dimension.}

The following theorem of Avramov and Buchweitz plays a central role in this paper. It should be noted that the hypothesis that $M$ has finite CI-dimension guarantees that $M$ has finite G-dimension (\ref{2.2.2}) and that the Tate homology modules $\widehat{\Tor}_i(M,-)$ are defined (\ref{2.1.4}) and (\ref{2.1.1}). A version of the equivalence of 
the first three conditions was shown by Jorgensen \cite[Thm.~2.1]{Jo99}. The final three conditions are much less fussy than the first three conditions; consequently, they serve as an advertisement for Tate homology.
\begin{theorem}\label{2.4*} {\bf \cite[Thm.~4.9]{AB00}}
 If $R$ is a Noetherian ring, and $M$ is a finitely generated  $R$-module of finite {\rm CI}-dimension, then for each $R$-module $N$ the following conditions are equivalent{\rm:}
\begin{enumerate}[\rm(i)]
\item $\Tor _i ^R (M, N) = 0$ for $\cx_R M + 1$ consecutive values of $i$ provided each of these values $i$ satisfies   $\CIdim_R M<i${\rm;}
\item $\Tor _i^ R (M, N) = 0$ for $0\ll i${\rm;}
\item $\Tor_i^ R (M, N) = 0$ for all $i$ with $\CIdim_RM<i${\rm;}
\item $\widehat\Tor_i^R (M, N) = 0$ for $\cx_R M + 1$ consecutive values of $i${\rm;}
\item $\widehat\Tor_i^R (M, N) = 0$  for $i \ll 0${\rm;} and
\item $\widehat\Tor_i^R (M, N) = 0$  for all integers $i$.
\end{enumerate}
\end{theorem}

\section{The two-step complete Tate complex associated to an ideal in a local ring.}\label{complete Tate}
Let $I$ be an  ideal  in a local ring $R$ with $\mu(K_1(I))\le \mu(I)$. In \ref{Define T} we define the minimal two-step complete Tate complex
$$ T:\quad
 \cdots \to T_1\to T_0 \to T_{-1}\to \cdots$$ for $I$ in $R$. In Corollary~\ref{Rigidity} we prove that $T$ is exact if and only if $I$ is a q.c.i.; furthermore, in this case, $T$ is the minimal complete resolution of $R/I$ by free $R$-modules.

\begin{data}\label{UniformData}Let $(R,\m)$ be a local ring,
$I$ be a proper 
 ideal of $R$ which is minimally generated by $\ff$ elements, $F$ be a 
 free $R$-module of rank $\ff$, and $\done:F\to R$ be an  
 $R$-module homomorphism with $\im (\done)=I$. 
Let $\bigwedge^\bullet F$ be the Koszul complex associated to $\done$, $\gog$ be the minimal number of generators of $\HH_1(\bigwedge^\bullet F)$, $G$ be a 
free $R$-module of rank $\gog$, and $$\ts\blop:G\to Z_1 (\bigwedge^\bullet F)$$ be an 
 $R$-module homomorphism with the property that the composition
\begin{equation}\label{composition}\ts G\xrightarrow{\blop}Z_1 (\bigwedge^\bullet F)\xrightarrow{\text{natural quotient map}}H_1 (\bigwedge^\bullet F)\end{equation} is a  
 surjection.\end{data}

\begin{remark-no-advance}\label{emphasize} We emphasize that the parameters $\ff$ and $\gog$ of Data~\ref{UniformData}
are equal to $\mu(I)$ and $\mu(K_1(I))$, respectively; and that the homomorphisms $\xi:F\to R$ and $G\to \HH_1(\bigwedge^\bullet F)$ of (\ref{composition}) are minimal homomorphisms.
\end{remark-no-advance}

\begin{definition}\label{Define T} Adopt the data of \ref{UniformData}. 
Let  $(-)^\vee$ be the functor $\Hom_{R}(-,\bigwedge^{\ff}F)$ and  
$$\ts P=\la  \bigwedge^\bullet F\t_{R} D_{\bullet} {G},\partial\ra,$$ be the 
two-step Tate complex associated to the data $(\done,\blop)$ of \ref{two-step}.
For $0\le i\le \ff-\gog$, define  an $R$-module homomorphism $$\ts \a_i:P_i\t_{R}\bigwedge^{\gog}G\to (P_{\ff-\gog-i})^\vee$$ as follows. Every component of $\a_i$ is zero except the component
\begin{equation}\notag\ts \bigwedge^i F\t_{R} \bigwedge^{\gog}G\to 
(\bigwedge^{\ff-\gog-i}F)^\vee; \end{equation}
and, if $\theta_i\in \bigwedge^i F$ and $\omega_{\gog}\in \bigwedge^{\gog}G$, then $$\ts\a_i(\theta_i\t \omega_{\gog}):
\bigwedge^{\ff-\gog-i}F\to\bigwedge^{\ff}F
$$ is the homomorphism which sends $\theta_{\ff-\gog-i}\in \bigwedge^{\ff-\gog-i}F$ to $$\ts(-1)^{\frac{i(i-1)}2+i\gog}\theta_i\w (\bigwedge^{\gog}\phi)\omega_{\gog}\w \theta_{\ff-\gog-i}\in 
\bigwedge^{\ff}F.$$ Let  $\ts \a:P\t_{R} \bigwedge^{\gog}G\to 
P^{\vee}[\goth g-\goth f]$ represent the picture
$$\xymatrix{
\cdots\ar[r]^(.35){\partial_{\goth f-\goth g+1}}&P_{\goth f-\goth g}\t_{R} \bigwedge^{\gog}G\ar[r]^(.65){\partial_{\goth f-\goth g}}\ar[d]^{\alpha_{\goth f-\goth g}}&\cdots\ar[r]^(.3){\partial_2}&P_1\t_{R} \bigwedge^{\gog}G\ar[r]^{\partial_1}\ar[d]^{\alpha_1}&P_0\t_{R}\bigwedge^{\gog}G\ar[r]\ar[d]^{\alpha_0}&0\\
0\ar[r]&(P_{0})^\vee\ar[r]^(.65){\partial_1^\vee}&\cdots\ar[r]^(.3){\partial_{\goth f-\goth g-1}^\vee}&(P_{\goth f-\goth g-1})^\vee\ar[r]^{\partial_{\goth f-\goth g}^\vee}&(P_{\goth f-\goth g})^\vee\ar[r]^(.65){\partial_{\goth f-\goth g+1}^\vee}&\cdots\ ,}$$
and let $T$ be the (formal) mapping cone of the picture $\a$. In other words, $T$ is the collection of maps
$$ T:\quad
 \cdots \xrightarrow{\tau_2} T_1\xrightarrow{\tau_1} T_0 \xrightarrow{\tau_0} T_{-1}\xrightarrow{\tau_{-1}} \cdots,$$
where $$T_i= \begin{matrix}P_i\t_R\bigwedge^{\gog}G\\\p\\P_{\ff-\gog-1-i}^\vee\end{matrix}\quad\text{and}\quad \tau_i=\bmatrix \partial_i&0\\\alpha_i&-\partial_{\ff-\gog-i}^\vee\endbmatrix,$$ for each integer $i$.
\end{definition}
\begin{remark-no-advance}\label{Define T*} It is not difficult to see that $\alpha$ is a map of complexes and that $\im \tau_i$ is contained in $\m T_{i-1}$ for each $i$.  It follows that  $T$ is a complex of free $R$-modules. We call $T$ the {\it  minimal two-step complete Tate complex associated to the data $(\xi,\phi)$.}\end{remark-no-advance}

\section{The  complete resolution associated to a pair of nested complete intersections.}\label{nested}
The main result in this section is Proposition~\ref{4.2}. 
 We prove that 
if $A\subseteq B$ is a pair of complete intersection ideals in the local ring $(\Ring,\M)$, with $A\subseteq \M B$, then the  ``minimal two-step complete Tate complex'' of Section~\ref{complete Tate}, for the ideal $B/A$ of the ring $\Ring/A$, is exact.

\begin{data}\label{dataJun8}Let $(\Ring,\M)$ be a local ring and $A\subseteq B$ be ideals in $\Ring$, with $A\subseteq \M B$. Assume that each of the ideals $A$ and $B$  is generated by a regular sequence. 
Let  $\goth g=\grade A$,  $\goth f=\grade B$,  
$\Fdagger$ and $\Gdagger$ be free $\Ring$-modules of rank $\goth f$ and $\goth g$, respectively, and  $\Xi:\Fdagger\to \Ring$ and $\Phi:\Gdagger\to \Fdagger$ be $\Ring$-module homomorphisms with $\im \Xi=B$ and $\im (\Xi\circ \Phi)=A$.
Let $\overline{\phantom{x}}$ represent the functor $(\Ring/A)\t_\Ring-$
and $(-)^\vee$ represent the functor $\Hom_{\Rbar}(-,\bigwedge^{\ff}\Fbar)$. 
 \end{data}

\begin{chunk}\label{Jul24}Adopt Data~\ref{dataJun8}. Let $\Pdagger$ be the two-step Tate complex associated to $(\Xibar,\ov{\Phi})$ as given in \ref{two-step}.  Tate \cite[Thm.~4]{T57} proved that $\Pdagger$ is a minimal resolution of 
$\Ring/B$ by free $\Ring/A$-modules. 
Define the map of complexes $\ts \a:\Pdagger\t_{\Rbar} \bigwedge^{\gog}\Gbar\to 
\Pdagger^{\vee}[\goth g-\goth f]$: 
$$\xymatrix{
\cdots\ar[r]^(.35){\partial_{\goth f-\goth g+1}}&\Pdagger_{\goth f-\goth g}\t_{\Rbar} \bigwedge^{\gog}\Gbar\ar[r]^(.65){\partial_{\goth f-\goth g}}\ar[d]^{\alpha_{\goth f-\goth g}}&\cdots\ar[r]^(.3){\partial_2}&\Pdagger_1\t_{\Rbar} \bigwedge^{\gog}\Gbar\ar[r]^{\partial_1}\ar[d]^{\alpha_1}&\Pdagger_0\t_{\Rbar}\bigwedge^{\gog}\Gbar\ar[r]\ar[d]^{\alpha_0}&0\\
0\ar[r]&(\Pdagger_{0})^\vee\ar[r]^(.65){\partial_1^\vee}&\cdots\ar[r]^(.3){\partial_{\goth f-\goth g-1}^\vee}&(\Pdagger_{\goth f-\goth g-1})^\vee\ar[r]^{\partial_{\goth f-\goth g}^\vee}&(\Pdagger_{\goth f-\goth g})^\vee\ar[r]^(.65){\partial_{\goth f-\goth g+1}^\vee}&\cdots\ ,}$$ exactly as was done in \ref{Define T}. The mapping cone $\Tdagger$ of $\a$ is the 
 minimal two-step complete Tate complex associated to the data $(\ov{\Xi},\ov{\Phi})$ as described in \ref{Define T}. \end{chunk}

\begin{remark-no-advance}The data of Section~\ref{nested} is analogous to  the data of Section~\ref{complete Tate} in the sense of \ref{emphasize}. In particular, $\mu(B/A)=\ff$, $\mu(K_1(B/A))=\gog$, and the homomorphisms $ \Xibar: \Fbar \to \Rbar$ and $\ov{\Gdagger} \to \HH_1(\bigwedge^\bullet \Fbar)$ are minimal. The inclusion $A\subseteq B$ guarantees that $\goth g\le \goth f$ because $\grade A$ and $\grade B$ are also equal to $\gog$ and $\ff$, respectively.
\end{remark-no-advance}

\begin{proposition}\label{4.2} Adopt the data of {\rm\ref{dataJun8}}. 
 Then the complex $\Tdagger$ of {\rm\ref{Jul24}} is the minimal complete resolution of $\Ring/B$ by free $\Rbar$-modules. 
\end{proposition}

Definition~\ref{eta} is used in our proof of  Proposition~\ref{4.2}.
\begin{definition}\label{eta} Retain the notation of \ref{dataJun8} and \ref{Jul24}. Let $\eta\in (\bigwedge^{\gog}\Gbar\t_{\Rbar}\Pdagger_{\ff-\gog})^\vee$ be the following homomorphism. The restriction of
$\eta$ 
 to 
$$\ts \sum\limits _{\gf{p+2q=\ff-\gog}{1\le q}}\ts \bigwedge^{\gog}\Gbar\t_{\Rbar}\bigwedge ^{p}\Fbar\t_{\Rbar}D_q\Gbar\to \bigwedge^{\ff}\Fbar$$ is equal to zero and the restriction
$\eta$  
to $\bigwedge^{\gog}\Gbar\t_{\Rbar}\bigwedge ^{\ff-\gog}\Fbar\t_{\Rbar}D_0\Gbar$
is given by 
$$\ts\eta(\omega_{\gog}\t \theta_{\ff-\gog})=(\bigwedge^{\gog}\phibar)(\omega_{\gog})\w \theta_{\ff-\gog},$$
for $\omega_{\gog}\in \bigwedge^{\gog}\Gbar$ and $\theta_{\ff-\gog}\in \bigwedge ^{\ff-\gog}\Fbar$.
\end{definition}

\medskip\noindent{\it Proof of Proposition~{\rm\ref{4.2}}.}
 Fix a generator $\omega_{\gog}$ of $\bigwedge^{\gog}\Gbar$. Observe that
$$\ts\a_0(\omega_{\gog})
=\eta(\omega_{\gog}\t -)
\in (\Pdagger_{\ff-\gog})^\vee;$$furthermore, $\eta(\omega_{\gog}\t -)$ is a co-cycle  in the complex $\Pdagger^\vee$.
It suffices to prove that
\begin{equation}\label{Jun15}\text{$\HH^{\ff -\gog}(\Pdagger^\vee)$
is 
generated by $\eta(\omega_g\t -)$.}\end{equation}
Indeed, 
it is shown in 
 \cite[2.5(4)]{AHS13}
that 
$$\HH^i(\Pdagger^\vee)\cong\begin{cases} 0&\text{if $i\neq\ff -\gog$, and}\\
\Ring/B&\text{if $i=\ff-\gog$}.\end{cases}$$
Thus, 
once  (\ref{Jun15}) is established, then 
the map of complexes $\a$ of {\rm\ref{Jul24}} is a quasi-isomorphism and the proof is complete.
We prove (\ref{Jun15}).  
Let $c_{\ff-\gog}$ be a  co-cycle in  $Z^{\ff-\gog}(\Pdagger^\vee)$. It follows that   $$\ts c_{\ff-\gog}=\sum\limits_{2p+q=\ff-\gog}\zeta_{p,q}\ ,$$with $\zeta_{p,q}\in (
D_p\Gbar\t_{\Rbar}\bigwedge^q\Fbar)^\vee$, and 
$$\ts \phibar^\vee(\zeta_{p,q})+\Xi^\vee(\zeta_{p+1,q-2})=0\quad\text{in}\quad (D_{p+1}\Gbar\t_{\Rbar}\bigwedge^{q-1}\Fbar)^\vee,$$ for all $(p,q)$ with $2p+q=\ff-\gog$. 

We first show 
that $\zeta_{0,\ff-\gog}$ is a scalar multiple of $\eta(\omega_\gog \t -)$. 

The natural quotient map $\Ring/A \to \Ring/B$ is a \qci homomorphism; see \cite[1.4]{AHS13}.
Hence, 
the $\Rbar$-module homomorphism $\phibar:\Gbar\to\Fbar$ induces an isomorphism of exterior algebras $\bigwedge^\bullet_{\Ring/B}(\Gdagger\t_\Ring\Ring/B) \to \HH_\bullet(\bigwedge^\bullet_{\Rbar}\Fbar)$. In particular, the class of the cycle $(\bigwedge^\gog\phibar)(\omega_{\gog})$ in $Z_{\gog}(\bigwedge^{\bullet}\Fbar)$ generates the homology module $\HH_{\gog}(\bigwedge^{\bullet}\Fbar)$. On the other hand, there is an isomorphism of complexes $$\ts(\bigwedge^\bullet \Fbar,\Xibar)\xrightarrow{\cong}  \big((\bigwedge^\bullet \Fbar)^\vee,\Xibar^\vee)\big)[-\ff],$$ which is induced by the map which sends 
$\theta_i$ in $\bigwedge^i\Fbar$  to $\pm \theta_i\w -$ in $(\bigwedge^{\ff-i}\Fbar)^\vee$. It follows that 
\begin{equation}\label{Jun15.2} \HH^i(\Pdagger^\vee)=0\quad\text{for $0\le i\le \ff-\gog-1$, and}
\end{equation}
the cohomology class of the co-cycle $(\bigwedge^\gog\phibar)(\omega_{\gog})\w -$ in $Z^{\ff-\gog}\big((\bigwedge^{\bullet}\Fbar)^\vee\big)$ generates the cohomology module $\HH^{\ff-\gog}\big((\bigwedge^{\bullet}\Fbar)^\vee\big)$. Thus, $\zeta_{0,\ff-\gog}=\lambda \eta(\omega_\gog\t -)$, for some $\lambda\in \Rbar$. We prove that 
$c_{\ff-\gog}-\lambda \eta(\omega_\gog\t -)$, which is equal to 
$$\sum\limits_{\gf{2p+q=\ff-\gog}{1\le p}}\zeta_{p,q},$$ is a boundary in $(\Pdagger^\vee)_{\ff-\gog}$. 

Let $p_0$ be the smallest index with $\zeta_{p_0,q_0}\neq 0$. Observe that $\zeta_{p_0,q_0}$ is a co-cycle in the complex 
\begin{equation}
\notag\ts 
(D_{p_0}\Gbar\t_{\Rbar}\bigwedge^{q_0-1}\Fbar)^\vee\xrightarrow{\Xibar^\vee}  
(D_{p_0}\Gbar\t_{\Rbar}\bigwedge^{q_0}\Fbar)^\vee\xrightarrow{\Xibar^\vee} (
D_{p_0}\Gbar\t_{\Rbar}\bigwedge^{q_0+1}\Fbar)^\vee.\end{equation}
This complex is exact because $\HH_i(\Pdagger^\vee)=0$ for $0\le i\le \ff-\gog-1$ 
(see, for example, (\ref{Jun15.2})) and $q_0=\ff-\gog-2p_0\le \ff-\gog-2$. 
Thus, $c_{\ff-\gog}-\lambda(\omega_\gog\t -)$ is congruent, mod $B^{\ff-\gog}(\Pdagger^\vee)$, to 
$$\sum\limits_{\gf{2p+q=\ff-\gog}{p_0'+1\le p}}\zeta_{p,q}',$$ for some $\zeta_{p,q}'\in (D_p\Gbar\t \bigwedge^q\Fbar)^\vee$. Iterate this procedure to conclude that $c_{\ff-\gog}$ and $\lambda\eta(\omega_\gog\t-)$ represent the same class in $\HH^{\ff-\gog}(\Pdagger^\vee)$. This completes the proof of (\ref{Jun15}). \hfil \qed

\section{The generic Tate construction.}\label{generic-data}

Given an (almost arbitrary) ideal $I$ in a local ring $R$, we produce a generic pair of nested complete intersection ideals that can be used in Theorem~\ref{main} to determine if $I$ is a \qci. There is a small restriction imposed on $I$; it must satisfy $$\mu(K_1(I))\le \mu(I),$$ where $K_1(I)$ is the first Koszul homology associated to a minimal generating set for $I$, as described in
\ref{2.1.9}.
 This hypothesis is benign, in the situation of interest, because if $I$ is a \qci, then $$\grade I=\mu(I)-\mu(K_1(I));$$ see \cite[Lem.~1.2]{AHS13}, and of course $\grade I$ is always non-negative.

\begin{data}\label{StandardData2}
Let $(R,\m)$ be a local ring,
$I$ be a proper 
 ideal of $R$ which is minimally generated by $\ff$ elements, $F$ be a 
 free $R$-module of rank $\ff$, and $\done:F\to R$ be a  
 $R$-module homomorphism with $\im (\done)=I$. 
Let $\bigwedge^\bullet F$ be the Koszul complex associated to $\done$, $\gog$ be the minimal number of generators of $\HH_1(\bigwedge^\bullet F)$, $G$ be a 
free $R$-module of rank $\gog$, and $$\ts\blop:G\to Z_1 (\bigwedge^\bullet F)$$ be an 
 $R$-module homomorphism with the property that the composition
$$\ts G\xrightarrow{\blop}Z_1 (\bigwedge^\bullet F)\xrightarrow{\text{natural quotient map}}H_1 (\bigwedge^\bullet F)$$ is a  
 surjection.
Assume $\gog\le \ff$.\end{data}

\begin{construction}\label{X3.4}Begin with the data of \ref{StandardData2}. Consider
the 
polynomial ring
 $$\Sym^R_\bullet (X_1\p X_2),$$ 
where $X_1$ and $X_2$ are the 
free $R$-modules $$X_1=F\quad\text{and}\quad X_2=F^*\t_R G.$$
Let $\M$ be the maximal ideal $$\M=\m+\sum_{1\le i}\Sym_i^R(X_1\p X_2)$$ of $\Sym^R_\bullet (X_1\p X_2)$; $\Rtil$ be the local ring $\big(\Sym^R_\bullet (X_1\p X_2)\big)_{\M}$; and $\widetilde{\phantom{a}}$ be the functor $-\t_R \Rtil$.
Define $$\Xi:\Ftil \to \Rtil\quad\text{and}\quad
\Phi:\Gtil \to\Ftil 
 $$ to be the 
 compositions
$$\Ftil=F \t_R \Rtil =X_1 \t_R \Rtil\xrightarrow{\mult}\Rtil$$
and 
\begin{align*}\Gtil=G \t_R \Rtil\xrightarrow{\ev^*\t 1\t 1}
F\t F^*\t G\t \Rtil
=F\t X_2 \t \Rtil\xrightarrow{1\t \mult}
F\t \Rtil =\Ftil, 
\end{align*} respectively.
  Define $\rho:\Rtil\to R$ to be the 
 $R$-algebra homomorphism induced
by 
\begin{align*}\rho(\theta_1)&=\done(\theta_1),&&&&\text{for $\theta_1\in F=X_1$, and}\\
\rho(\Theta_1\t g)&=(\Theta_1\circ \blop)(g),&&&&\text{for $\Theta_1\t g\in F^*\t G=X_2$.}
\end{align*} \end{construction}
\begin{Remark} Observe that $\rho(\M)\subseteq \m$. Indeed, 
$$\ts \rho(X_1)\subseteq \im(\done)\subseteq I\subseteq \m\quad\text{and}\quad \rho(X_2)\subseteq I_1\big(Z_1(\bigwedge^\bullet F)\big)\subseteq \m.$$
The final inclusion holds because the rank of $F$ is the minimal number of generators of $I$; hence, $\done:F\to R$ is a minimal homomorphism.
\end{Remark}

\begin{proposition}\label{Jun17-843}Given the data  
of  {\rm\ref{StandardData2}}, apply Construction~{\rm\ref{X3.4}} to produce $\Rtil$, $\Xi$, $\Phi$, and $\rho$. Let $B$ be  the image of $\Xi$ in $\Rtil$ and  $A$ be the image of $\Xi\circ \Phi$ in $\Rtil$.
The following statements hold{\rm:}
\begin{enumerate}[\rm(a)]
\item\label{X3.1.b} $B$ is generated by a regular sequence on $\Rtil$ of length $\ff$,
\item\label{X3.1.a} $A$ is generated by a regular sequence on $\Rtil$ of length $\gog$,
\item\label{X3.1.f} $BR=\im \done$,
\item\label{X3.1.g} $AR=0$,
\item\label{4.3.e} the $\Rtil$-module homomorphisms $\Xi:\Ftil\to \Rtil$ and $\Phi:\Gtil\to \Ftil$ are minimal,
\item\label{4.3.f} the $R$-module homomorphism $
\Xi\t 1:
\Ftil\t_{\Rtil}R\to \Rtil\t_{\Rtil}R$ is equal to $\xi:F\to R$,  and
\item\label{4.3.g} the $R$-module homomorphism $\Phi\t 1:
\Gtil\t_{\Rtil}R\to \Ftil\t_{\Rtil}R$ is equal to $\phi:G\to F$.
\end{enumerate}\end{proposition}\begin{Remark}
The homomorphism $\rho$ makes $R$   an $\Rtil$-algebra; this $\Rtil$-algebra structure on $R$ is used in (\ref{X3.1.f}), (\ref{X3.1.g}), (\ref{4.3.f}), and (\ref{4.3.g}).
\end{Remark}

\begin{proof} {\bf(\ref{X3.1.b})} The ideal $B$ is generated by $\ff$ distinct indeterminates; these generators form a regular sequence.

\smallskip\noindent {\bf (\ref{X3.1.a})} The ideal $A$ is generated by the entries of the product $\mathbf b\mathbf C$, where $\mathbf b$ and $\mathbf C$ are matrices of distinct indeterminates, $\mathbf b$ has shape $1\times \ff$, $\mathbf C$ has shape $\ff\times \gog$, and $\gog\le \ff$. These generators form a regular sequence; see, for example, \cite[6.13]{N62}.

\smallskip\noindent {\bf (\ref{X3.1.f})} The ideal $B$ of $\Rtil$ is generated by $\Xi(F)=X_1$; so, $\rho(B)$ is generated by $$(\rho\circ \Xi)(\Ftil)=\done(F).$$

\smallskip\noindent {\bf (\ref{X3.1.g})} Let $(\{x_i\},
\{x_i^*\})$ be a pair of dual bases for $F$ and $F^*$, respectively. The ideal
$A$ of $\Rtil$ is generated by 
$$\Big\{ \sum_ix_i(x_i^*\t g)\in X_1\cdot X_2\subseteq \Rtil\ \Big\vert \ g\in G\Big\}.$$It follows that $\rho(A)$ is generated by $$\sum_i \done(x_i)\cdot x_i^*(\blop(g))=(\done\circ \blop)(g)\in \done\big( Z_1(\ts \bigwedge^\bullet F,\done)\big)=0.$$

\smallskip\noindent {\bf (\ref{4.3.e})} If one expresses either of these maps as a matrix, then the entries of this matrix form a regular sequence. It follows that the kernel of the map is in $\M$.

\smallskip\noindent {\bf (\ref{4.3.f}) and (\ref{4.3.g})} The composition
$$R\xrightarrow{\text{inclusion}} \Rtil \xrightarrow{\rho}R$$ is the identity map. It follows that $\Rtil\t_{\Rtil}R\cong R$. The assertions are now obvious. 
\end{proof}

\section{The main theorem.}

Data~\ref{5.1} has 3 parts. Part (\ref{5.1.a}) concerns an ideal $I$ in a local ring $R$; this part of the data is exactly the same as Data~\ref{UniformData}, except that hypothesis~\ref{HYP} has now been added. Part (\ref{5.1.b})  is about a pair of nested complete intersection ideals $A\subseteq B$ in a local ring $\Ring$. Finally, part (\ref{5.1.c}) is about   a surjection $\rho:\Ring\to R$ which carries $A$ to $0$ and $B$ to $I$. Proposition~\ref{Jun17-843} guarantees that for every ideal $I$ which satisfies the hypotheses of (\ref{5.1.a}), the rest of Data~\ref{5.1} can be created generically. The subsequent results in the paper may be applied to the generic data built in  Proposition~\ref{Jun17-843} or any other data which satisfies the hypotheses of Data~\ref{5.1}.
\begin{data}\label{5.1} \begin{enumerate}[\rm(a)]
\item\label{5.1.a}
Let $(R,\m)$ be a local ring,
$I$ be a proper 
 ideal of $R$ which is minimally generated by $\ff$ elements, $F$ be a 
 free $R$-module of rank $\ff$, and $\done:F\to R$ be an  
 $R$-module homomorphism with $\im (\done)=I$. 
Let $\bigwedge^\bullet F$ be the Koszul complex associated to $\done$, $\gog$ be the minimal number of generators of $\HH_1(\bigwedge^\bullet F)$, $G$ be a 
free $R$-module of rank $\gog$, and $$\ts\blop:G\to Z_1 (\bigwedge^\bullet F)$$ be an 
 $R$-module homomorphism with the property that the composition
\begin{equation}\label{composition*}\ts G\xrightarrow{\blop}Z_1 (\bigwedge^\bullet F)\xrightarrow{\text{natural quotient map}}H_1 (\bigwedge^\bullet F)\end{equation} is a  
 surjection.
Assume 
\begin{equation}\label{HYP}0\le \ff-\gog\le \grade I.\end{equation}

\item\label{5.1.b}Let $\Ring$ be a local ring and $A\subseteq B$ be ideals in $\Ring$, each of which is generated by a regular sequence. 
Let  $\goth g=\grade A$,  $\goth f=\grade B$,  
$\Fdagger$ and $\Gdagger$ be free $\Ring$-modules of rank $\goth f$ and $\goth g$, respectively, and  $\Xi:\Fdagger\to \Ring$ and $\Phi:\Gdagger\to \Fdagger$ be minimal $\Ring$-module homomorphisms with $\im \Xi=B$ and $\im (\Xi\circ \Phi)=A$. Let $\Pdagger$ and $\Tdagger$ be the Tate resolution and the complete Tate resolution of $\Ring/B$ by free $\Ring/A$-modules as described in \ref{Jul24} and \ref{4.2}.

\item\label{5.1.c}Let $\rho:\Ring\to R$ be a surjective ring homomorphism with $A\subseteq \ker \rho$. Assume that $\Fdagger\t_{\Ring} R=F$, $\Gdagger\t_{\Ring} R=G$, the composition 
$$F=\Fdagger\t_{\Ring} R\xrightarrow{\Xi\t 1} \Ring\t_{\Ring}R=R$$ is $\done$ and the composition $$G=\Gdagger\t_{\Ring}R\xrightarrow{\Phi \t_{\Ring}1}\Fdagger\t_{\Ring}R=F$$ is $\blop$.  
\end{enumerate}\end{data}
\begin{Remark} We use Data~\ref{5.1} as we state and prove conditions which are equivalent to the statement ``$I$ is a \qci''. Recall from \cite[Lem.~1.2]{AHS13} that if $I$ is a \qci, then $\grade I=\ff-\gog$; and therefore, the inequality (\ref{HYP}) holds automatically in this case.
\end{Remark}

\begin{proposition}\label{4.5}Adopt the data of {\rm\ref{5.1}}.
Then the following statements hold{\rm:}
\begin{enumerate}[\rm(a)]
\item\label{4.5.a} $\Pdagger$ and $\Tdagger$ are the minimal resolution and the minimal complete resolution of $\Ring/B$ by free $\Ring/A$-modules, respectively{\rm;}
\item\label{4.5.b} $\Pdagger\t_{\Ring/A}R$ and $\Tdagger\t_{\Ring/A}R$ are the minimal two-step Tate complex and the minimal two-step complete Tate complex associated to the data $(\done,\blop)$ in the sense of {\rm\ref{two-step}} and {\rm\ref{Define T*}}, respectively{\rm;} and 
\item\label{4.5.c} the following statements are equivalent{\rm:}
  \begin{enumerate}[\rm(i)]
\item\label{4.5.c.i} $I$ is a \qci ideal of $R${\rm;}
\item\label{4.5.c.ii} $\Pdagger\t_{\Ring/A}R$ is a resolution of $R/I$ by free $R$-modules{\rm;} and
\item\label{4.5.c.iii} $\Tor_i^{\Ring/A}(\Ring/B,R)=0$ for all positive $i${\rm;}\end{enumerate}
\item\label{4.5.d} the following statements are equivalent{\rm:}
\begin{enumerate}[\rm(i)]\item\label{4.5.c.iv} $\Tdagger\t_{\Ring/A}R$ is a complete resolution of $R/I$ by free $R$-modules{\rm;}  and 
\item\label{4.5.c.v} $\widehat{\Tor}_i^{\Ring/A}(\Ring/B,R)=0$ for all integers $i$.
\end{enumerate}
\end{enumerate}
\end{proposition}

\begin{proof}Assertion  (\ref{4.5.a}) is established in \cite[Thm.~4]{T57} (for $\Pdagger$) and Proposition~\ref{4.2} (for $\Tdagger$).
Assertion~(\ref{4.5.b}) follows from the definition of $\rho$. 
Assertions (\ref{4.5.c.i}) and (\ref{4.5.c.ii}) are equivalent because of (\ref{2.2.4}). Assertions (\ref{4.5.c.ii}) and (\ref{4.5.c.iii}) are equivalent because the homology of $\Pdagger\t_{\Ring/A}R$ is $\Tor_\bullet^{\Ring/A}(\Ring/B,R)$. Similarly, assertions (\ref{4.5.c.iv}) and (\ref{4.5.c.v}) are equivalent because the homology of $\Tdagger\t_{\Ring/A}R$ is $\widehat{\Tor}_\bullet^{\Ring/A}(\Ring/B,R)$.
\end{proof}

The main result of the paper, Theorem~\ref{main}, is an extension of Proposition~\ref{4.5}.(\ref{4.5.c}), by way of Theorem~\ref{2.4*}. The two-step Tate complex associated to the data $(\done,\blop)$ exhibits significant rigidity (see assertions (\ref{main.ii}) and (\ref{main.iii}) of Theorem~\ref{main} and also Corollary~\ref{Rigidity}) and one can use  Tate homology  in place of ordinary homology  when determining if $I$ is a \qci (see assertions (\ref{main.v}), (\ref{main.vi}), and (\ref{main.vii}) of Theorem~\ref{main}). Furthermore, if $I$ is a \qci, then an explicit complete resolution for $R/I$ is given.

The first step in the transition from Proposition~\ref{4.5} to Theorem~\ref{main} is given in Observation~\ref{4.6}.

\begin{observation}\label{4.6}Adopt the data of  {\rm\ref{5.1}}. The following statements hold{\rm:} 
\begin{enumerate}[\rm(a)]
\item $\CIdim_{\Ring/A}(\Ring/B)=\ff-\gog$, and
\item $\cx_{\Ring/A}(\Ring/B)=\gog$. 
\end{enumerate}
\end{observation}
\begin{proof} First consider the quasi-deformation 
$$\xymatrix{\Ring/A\ar[rrr]^{\ \ =\ \ }&&& \Ring/A
\ar@{<-}[rrr]^{\text{natural quotient map}}&&&
\Ring.}$$
Observe that $$\pd_{\Ring}(\Ring/B)-\pd_{\Ring} (\Ring/A)=\ff-\gog.$$ It follows from (\ref{2.3.1}) that $\CIdim_{\Ring/A}(\Ring/B)$ is finite (and at most $\ff-\gog$). Furthermore, it follows from (\ref{2.2.2}) that 
\begingroup\allowdisplaybreaks\begin{align*}\CIdim_{\Ring/A}(\Ring/B)={}&\depth (\Ring/A)-\depth_{\Ring/A}(\Ring/B)\\{}={}&\depth (\Ring/A)-\depth(\Ring/B)\\{}={}&(\depth \Ring-\gog)-(\depth \Ring-\ff)=\ff-\gog.\end{align*}\endgroup
At this point, \ref{2.4} guarantees that $\cx_{\Ring/A} \Ring/B$ is the order of the pole of the Poincar\'e series $P_{\Ring/B}^{\Ring/A}$. The complex  $\Pdagger$ is the minimal resolution of 
$\Ring/B$ by free $\Ring/A$ modules; thus $P_{\Ring/B}^{\Ring/A}=\frac{(1+t)^\ff}{(1-t^2)^\gog}$. It follows that $\cx_{\Ring/A} \Ring/B=\gog$.
\end{proof}

\begin{theorem}\label{main} Adopt the data of {\rm\ref{5.1}}. 
Then the following statements are equivalent{\rm:}
\begin{enumerate}[\rm(i)]
\item \label{main.ii}$\Tor _i ^{\Ring/A} (\Ring/B, R) = 0$ for $\gog + 1$ consecutive values of $i$ with $0<i${\rm;}
\item \label{main.iii}$\Tor _i^{\Ring/A} (\Ring/B, R) = 0$ for $0\ll i${\rm;}
\item \label{main.iv}$\Tor_i^{\Ring/A} (\Ring/B, R) = 0$ for all $i$ with $0<i${\rm;}
\item \label{main.v}$\widehat\Tor_i^{\Ring/A} (\Ring/B, R) = 0$ for $\gog + 1$ consecutive values of $i${\rm;}
\item \label{main.vi}$\widehat\Tor_i^{\Ring/A} (\Ring/B, R) = 0$  for $i \ll 0${\rm;} 
\item \label{main.vii}$\widehat\Tor_i^{\Ring/A} (\Ring/B, R) = 0$  for all integers $i${\rm;} and
\item\label{main.i}  $I$ is a \qci ideal of $R$.
\end{enumerate}
Furthermore, if the above statements hold, then $\Pdagger\t_{\Ring/A}R$ and $\Tdagger\t_{\Ring/A}R$ are the minimal resolution and the minimal complete resolution of $R/I$ by free $R$-modules, respectively.
\end{theorem}

\begin{proof}We saw in Proposition~\ref{4.5}.(\ref{4.5.c}) that 
\begin{equation}\label{real5.4.1} \text{(\ref{main.i})} \iff \text{(\ref{main.iv})}.\end{equation}
 Apply Theorem~\ref{2.4*} with $R$ replaced by $\Ring/A$, $M$ by $\Ring/B$, and $N$ by $R$. Use the results from Observation~\ref{4.6}:
$$\CIdim_{\Ring/A}\Ring/B=\ff-\gog\quad\text{and}\quad \cx_{\Ring/A}\Ring/B=\gog.$$It follows that the statements 
\begin{equation}\label{equiv}\text{(\ref{main.ii}$'$),\quad (\ref{main.iii}),\quad (\ref{main.iv}$'$),\quad(\ref{main.v}),\quad(\ref{main.vi}),\quad(\ref{main.vii})}\end{equation}
 are equivalent, where (\ref{main.ii}$'$) and (\ref{main.iv}$'$) are

\medskip\noindent (\ref{main.ii}$'$)\quad $\Tor _i ^{\Ring/A} (\Ring/B, R) = 0$ for $\gog + 1$ consecutive values of $i$ provided each of these values $i$ satisfies   $\ff-\gog<i${\rm;} and 

\medskip\noindent(\ref{main.iv}$'$)\quad  $\Tor_i^{\Ring/A} (\Ring/B, R) = 0$ for all $i$ with $\ff-\gog<i${\rm.}

\medskip \noindent
It is clear that 
$$\text{(\ref{main.iv})}\implies \text{(\ref{main.iv}$'$)\quad and\quad(\ref{main.ii})}\implies
\text{(\ref{main.ii}$'$)}.$$
To complete the proof, we show 
\begin{equation}\label{STS}\text{(\ref{main.iv}$'$)}\implies \text{(\ref{main.iv})\quad and\quad(\ref{main.ii}$'$)}\implies
\text{(\ref{main.ii})}.\end{equation}

Let $r=\ff-\gog$ and $\underline{x}=x_1,\dots,x_r$ be the beginning of a minimal generating set for $I$ which is also a 
 regular sequence in $I$ on $R$. (Hypothesis~(\ref{HYP}), together with the prime avoidance lemma, guarantees that $\underline{x}$ exists.) Let $R'$, $I'$,
$\ff'$ and $\gog'$ denote 
  $R/(\underline{x})$, $I/(\underline{x})$, $\mu(I')$ and $\mu\big(\HH_1(\bigwedge^\bullet (F\t_RR'))\big)$, respectively. 
The fact that $\underline{x}$ begins a minimal generating set for $I$ ensures that $\ff'=\ff-r$.  Create the data $(\done',\blop')$ for $I'$ in $R'$ as described in Proposition~\ref{2.2.5}.
The proof of Proposition~\ref{2.2.5} demonstrates that there is a quasi-isomorphism from the Koszul complex associated to $\done$ to the Koszul complex associated to $\done'$. It follows that $\gog'=\gog$. 
The statement of Proposition~\ref{2.2.5} asserts  that there is a
quasi-isomorphism from  the two-step Tate complex for $I$ in $R$ 
to the two-step Tate complex for $I'$ in $R'$. 
On the other hand, we know from Proposition~\ref{4.5}.(\ref{4.5.b}) that $\Pdagger\t_{\Ring/A}R$ is the minimal two-step Tate complex for $I$ in $R$ and 
$\Pdagger\t_{\Ring/A}R'$ is the minimal two-step Tate complex for the ideal $I'$ of $R'$. Thus, 
\begin{equation}\notag\HH_i(\Pdagger\t_{\Ring/A}R)\cong  
\HH_i(\Pdagger\t_{\Ring/A}R'),\quad\text{for all $i$;}\end{equation}hence,
\begin{equation}\label{qisom}
\Tor _i^{\Ring/A} (\Ring/B, R)\cong 
\Tor _i^{\Ring/A} (\Ring/B, R'),\quad\text{for all $i$.}
\end{equation}

\medskip
We prove (\ref{STS}). 
 Assume that either (\ref{main.ii}$'$) or (\ref{main.iv}$'$) holds for $I$. It follows from (\ref{equiv}) that (\ref{main.iii}) holds for $I$.  
Apply 
(\ref{qisom}) to see   that (\ref{main.iii}) holds for $I'$. Hence, (\ref{main.iv}$'$) holds for $I'$ by (\ref{equiv}), again. On the other hand, (\ref{main.iv}$'$) for $I'$ is the same as (\ref{main.iv}) for $I'$ because $$\ff'-\gog'=(\ff-r)-\gog=0.$$ Use (\ref{qisom}), again, to see   that (\ref{main.iv}) holds for $I$. It is clear that (\ref{main.iv}) implies (\ref{main.ii}).
\end{proof}

Corollaries~\ref{May19} and \ref{every} are reformulations of (\ref{main.i}) implies (\ref{main.iv}) from Theorem~\ref{main}. Corollary~\ref{May19} is easier to apply than the full statement of Theorem~\ref{main}.

\begin{corollary}\label{May19} If $I$ is a \qci ideal in a local ring $R$, then there exists a local ring $\Ring$ and ideals $A\subseteq B\cap C$ in $\Ring$ such that 
\begin{enumerate}[\rm(a)]
\item $A$ is generated by a regular sequence of length $\mu(I)-\grade I$, 
\item $B$ is generated by a regular sequence of  length $\mu(I)$, 
\item $R=\Ring/C$,
\item $\Tor_i^{\Ring/A}(\Ring/B,R)=0$, for $1\le i$,
 and 
\item $BR=I$.
\end{enumerate}
\end{corollary}
\begin{proof}
Apply Theorem~\ref{main} to the generic data built in Proposition~\ref{Jun17-843} for the ideal $I$ in $R$.
\end{proof}
The following observation-definition has been adapted from \cite[8.7]{AHS13}.
\begin{chunk} Let 
$\rho:Q\to R$ be a surjective homomorphism of
Noetherian local rings,
and  $\mcI$ be a \qci ideal of $Q$.
If
$\Tor_i^Q(Q/\mcI,R)=0$ for $1\le i$, then 
$\mcI R$ is a \qci ideal of $R$. 
Furthermore, one says that $\mcI R$ is obtained from $\mcI$ {\it by flat base change}. \end{chunk}

\begin{corollary}\label{every}Every \qci ideal in a local Noetherian ring is obtained from a pair of nested complete intersection ideals by way of a flat base change.\end{corollary}

\begin{proof} Apply Corollary~\ref{May19} with $Q=\Ring/A$, $\mcI=B/A$, and $\ker \rho=C/A$. \end{proof}

\begin{examples}\label{Ex} Examples (\ref{Ex.a}) and  (\ref{Ex.b}) were the well-understood examples of \qci ideals as described in \cite{KSV14}. On the other hand, the Example (\ref{Ex.d})  is also given,  but was not well-understood, in \cite{KSV14}.  \begin{enumerate}[\rm(a)]\item\label{Ex.a} 
If the ideal $I$ is generated by a regular sequence in the local ring $R$, then, in the language of Corollary~\ref{May19},  one can take $\Ring=R$, $A=C=0$, and $B=I$. 

\item\label{Ex.b} If $a$ and $b$ are a pair of exact zero divisors in the local ring $(R,\m)$ and $I$ is generated by $a$, then, in the language of Corollary~\ref{May19}, one can take $$\Ring=R[x_1,x_2]_{(\mathfrak m,x_1,x_2)}, \quad A=(x_1x_2),\quad\text{and}\quad B=(x_1).$$ Define $\rho:\Ring\to R$ to be the $R$-algebra homomorphism with  $\rho(x_1)=a$, and $\rho(x_2)=b$. 
Thus, $C=(x_1-a,x_2-b)$. 

\item\label{Ex.d} It was not possible to explain the
 \qci ideal $I$ of the ring $R$ in \cite[Sect.~4]{KSV14} using any of the techniques that appeared in \cite{KSV14}.

Let $\kk$ be a field, 
$\Ring$ be the polynomial ring $$\Ring=\kk[x_1,x_2,x_3,x_4,x_5],$$  $C$ be the ideal
$$C=(
x_1^2-x_2x_3,\ x_2^2-x_3x_5,\ x_3^2-x_1x_4,\ x_4^2,\ x_5^2,\ x_3x_4,\ x_2x_5,\ x_4x_5)$$
of $\Ring$, $f_1$ and $f_2$ be the elements $f_1=x_1+x_2+x_4$ and $f_2= x_2+x_3+x_5$ of $\Ring$,
 $R$ be the ring $R=\Ring/C$   and  
$I$ be the ideal $(f_1,f_2)R$ of $R$.  

The following explanation of $I$ lead to Corollary~\ref{every}. 
Define $A$ to be  the ideal of $\Ring$ generated by the entries of the product
$$\left[\begin{matrix} f_1&f_2\end{matrix}\right]\left[\begin{matrix}
   x_1-x_2  &x_4    \\-x_3+x_4+2x_5 &x_2-x_3-x_4 \end{matrix}\right]$$
and $B=(f_1,f_2)\Ring$. Observe that $A\subseteq B$ are complete intersections. Observe further that 
$\pd_{\Ring/A}R$ is finite; indeed, the
 minimal resolution of $R$ is
$$\textstyle  
0\to \frac {\Ring}{A}(-4)^3 \to \frac {\Ring}{A}(-3)^8\to \frac {\Ring}{A}(-2)^6\to \frac {\Ring}{A}.$$Thus,
$$\Tor_i^{\Ring/A}(\Ring/B, R)=0,\quad\text{for $0\ll i$};$$
Apply (\ref{main.iii}) implies (\ref{main.iv}) from Theorem~\ref{main} to conclude that $$\Tor_i^{\Ring/A}(\Ring/B, R)=0,\quad\text{for $0< i$};$$
hence $I$ is obtained from the \qci ideal $B(\Ring/A)$ by way of flat base change. (In this example, one should localize as needed.)
\item We sketch a coordinate dependent argument for Corollary~\ref{every} in the general case. Let 
 $b_1,\dots,b_\ff$ be a minimal  generating set for the \qci ideal $I$ in the local ring $(R,\m)$, let  $(E,\partial)$ be the Koszul complex on this generating set,  and $v_1,\dots,v_\ff$ be a basis for $E_1$ with $\partial (v_i)=b_i$.

 Consider a set of cycles  $$z_j=\sum_{i=1}^\ff c_{i,j}v_i,$$in $E_1$, with $c_{i,j}\in \m$ and $1\le j\le \gog$, such that the homology classes $$\{\operatorname{cls}(z_j)\mid 1\le j\le \gog\}$$ minimally generate $\HH_1(E)$. 
According to \cite[1.2]{AHS13}, $\gog\le \ff$; indeed, $\grade_R(I)$ is equal to $\ff-\gog$.
Let $$\{\widetilde{b_i}\mid 1\le i\le \ff\}\cup \{\widetilde{c_{i,j}}\mid 1\le i\le \ff, 1\le j\le \gog\}$$ represent new indeterminates,
 $\mathfrak M$  be the maximal homogeneous ideal of the polynomial ring
 $$R[\{\widetilde{b_i}\}\cup \{\widetilde{c_{i,j}}\}],$$ $\widetilde{R}$ be  the local ring
\begin{equation}\label{.poly}
\notag
\widetilde{R}=R[\{\widetilde{b_i}\}\cup\{\widetilde{c_{i,j}}\}]_{\mathfrak M},\end{equation}
 and  $\rho:\widetilde{R}\to R$ be the surjective local $R$-algebra homomorphism with $\rho(\widetilde{b_i})=b_i$ and $\rho(\widetilde{c_{i,j}})=c_{i,j}$. (The $b_i$ are in the maximal ideal of $R$ because $I$ is a proper ideal of $R$; the $c_{i,j}$ are in the maximal ideal of $R$ because $b_1,\dots,b_\ff$ minimally generate $I$.)

Consider the ideals $A\subseteq B$ in $\widetilde R$, 
$$A=\big(\big\{\sum_i \widetilde{b_i}\widetilde{c_{i,j}}\mid 1\le j\le \gog\big\}\big)\quad\text{and}\quad B=(\widetilde{b_1},\dots,\widetilde{b_\ff}).$$ 
The ideals $A$ and $B$ are both complete intersections; $B/A$ is a \qci ideal of $R/A$; and the two step   Tate complex
$$\Pdagger=(\widetilde{R}/A){<}V_1,\dots V_\ff, W_1,\dots, W_\gog\mid \partial(V_i) = \widetilde{b_i}, \partial(W_j)= \sum_i \widetilde{c_{i,j}} V_i{>}$$ is a resolution of $\widetilde{R}/B$ by free $(\widetilde{R}/A)$-modules. This notation is explained in Remark~\ref{TateConstruction}. 

Notice that $$\rho(\sum_i \widetilde{b_i}\widetilde{c_{i,j}})
=\sum_i b_ic_{i,j}=\partial(\sum_i c_{i,j}v_i)=\partial(z_j)=0;$$ so $R$ is a $\widetilde{R}/A$-algebra. 
Notice also that $R\otimes_{\widetilde{R}/A} \Pdagger$ is the two-step Tate complex $$R\otimes_{\widetilde{R}/A} \Pdagger= R{<}v_1,\dots v_\ff, w_1,\dots, w_\gog\mid \partial(v_i) = b_i, \partial(w_j)= z_j{>},$$ which is a resolution of $R/I$ by free $R$-modules. 
It follows that
$$\Tor_i^{\widetilde{R}/A}(R,\widetilde{R}/B)=0,\quad\text{for $1\le i$}.$$
Thus, the \qci ideal is obtained from a pair of nested complete intersection ideals  by way of a flat base change. 
\end{enumerate} \end{examples}

\section{Application: Rigidity of the two-step Tate complex and the two-step complete Tate complex.}
In this section, we record  our rigidity result Corollary~\ref{Rigidity} 
and compare it
 to  the rigidity result of Jason Lutz.

\begin{corollary} \label{Rigidity}
Let $I$ be an ideal  in a local ring $R$.
Assume that $$0\le \mu(I)-
\mu(K_1(I))\le \grade I.$$
Let $P$ be the  minimal two-step Tate complex associated to $I$ in $R$ as described in Definition~{\rm\ref{two-step}} and $T$ be the minimal two-step complete Tate  complex associated to $I$ in $R$ as described in Definition~{\rm\ref{Define T}}.
 The following statements are equivalent{\rm:} 
\begin{enumerate}[\rm(a)]
\item \label{ain.ii}$\HH_i(P)= 0$ for $\mu(K_1(I)) + 1$ consecutive values of $i$ with $0<i${\rm;}
\item \label{ain.iii}$\HH_i(P)= 0$ for $0\ll i${\rm;}
\item \label{ain.iv}$\HH_i(P)= 0$ for all $i$ with $0<i${\rm;}
\item \label{ain.v}$\HH_i(T) = 0$ for $\mu(K_1(I)) + 1$ consecutive values of $i${\rm;}
\item \label{ain.vi}$\HH_i(T)  = 0$  for $i \ll 0${\rm;} 
\item \label{ain.vii}$\HH_i(T) = 0$  for all integers $i${\rm;} and
\item\label{ain.i}  $I$ is a \qci ideal of $R$.
\end{enumerate}
Furthermore, if the above statements hold, then $P$ and $T$ are the minimal resolution and the minimal complete resolution of $R/I$ by free $R$-modules, respectively.
\end{corollary}
\begin{proof}
Apply Theorem~\ref{main} to the generic data built in Proposition~\ref{Jun17-843} for the ideal $I$ in $R$.
Recall from Proposition~\ref{4.5}.(\ref{4.5.b}) that $\Pdagger\t_{\Ring/A}R=P$ and $\Tdagger\t_{\Ring/A}R=T$.
\end{proof}

\begin{corollary}\label{7.5} If $I$ is a \qci ideal in a local ring $R$ and $T$ is the minimal complete resolution of $R/I$ by free $R$-modules, then $T$ is isomorphic to $\Hom_R(T,R)$.
\end{corollary}

\begin{proof} The minimal complete resolution of $R/I$ is shown in Corollary~\ref{Rigidity} to be the  
minimal two-step complete Tate complex associated to $I$ in $R$
of \ref{Define T}. This complex has the stated property.\end{proof}

Theorem~\ref{Lutz} is Jason Lutz's rigidity result. 

\begin{theorem} {\rm(\cite[Thm.~3.1]{L16})} \label{Lutz} Let $I$ be an ideal in a local ring $R$ and let $P$ be the two-step Tate complex for $I$. 
Assume $
\mu(K_1(I))\le \mu(I)-\grade(I)$.
 If 
$\HH_i (P) = 0$ for $q\le  i\le q+\mu(I)-\grade(I)$, 
for some integer $q$, with $2\le q$,
then $I$ is a quasi-complete intersection.
\end{theorem}

\begin{remark}\label{7.8} The  results \ref{Rigidity} and \ref{Lutz}  agree in that they both show that if $\HH_i(P)=0$ for an  appropriate collection of consecutive integers $i$, then $I$ is a \qci. The two results differ in three aspects:\begin{enumerate}[\rm(1)]
 \item\label{6.2.1} the technical assumption on the acceptable inequalities relating $\mu(I)$, 
$\mu(K_1(I))$, and $\grade I$ appear to be different;
\item\label{6.2.2} our result allows one to be the beginning of the band of vanishing homology, but Lutz insists that band begin at some integer which is at least two; and
\item\label{6.2.3} our result needs 
$\mu(K_1(I))+1$ consecutive integers $i$ with $\HH_i(P)=0$; whereas Lutz's result needs $\HH_i(P)$ to vanish for $\mu(I)-\grade(I)+1$ consecutive values of $i$. 
\end{enumerate}
Notice, however, that if $I$ is a \qci then $\grade I=\mu(I)-
\mu(K_1(I))$, see \cite[1.2]{AHS13}. In this case, both technical assumptions from (\ref{6.2.1}) hold and the  parameters from (\ref{6.2.3}) are equal.
\end{remark}

\section{Application: The dimension theorem for quasi-homogeneous \qci ideals.}
In this section we reprove \cite[Thm.~4.1(c)]{AHS13}   using different methods. The ring $R$ in the following result is non-negatively graded over a field; this ring does not have to be standard graded.

\begin{proposition} Let $R=\bigoplus_{0\le i} R_i$ be a local graded ring with $R_0$ equal to a field. If $I$ is a homogeneous \qci ideal in  $R$, then
$$\grade I=\dim R-\dim R/I.$$ \end{proposition}

\begin{proof} Apply Corollary~\ref{May19} and identify 
a local ring $\Ring$ and ideals $A\subseteq B\cap C$ in $\Ring$ with $A$ generated by a regular sequence of length $\mu(I)-\grade(I)$, $B$ generated by a regular sequence of length $\mu(I)$, $\Tor_i^{\Ring/A}(\Ring/B,\Ring/C)=0$ for $1\le i$, $R=\Ring/C$, and $I=BR$.
If $J$ is generated by a regular sequence of length $r$ in  a local ring $Q$, then
$$\dim Q/J=\dim Q-r=\dim Q-\grade J.$$
Therefore, $$\dim \Ring/A=\dim \Ring-\mu(I)+\grade I;\quad
\dim \Ring/B=\dim \Ring-\mu(I);$$ 
and $\dim \Ring/A-\dim \Ring/B=\grade I$.

On the other hand, the fact that $\Tor_i^{\Ring/A}(\Ring/B,\Ring/C)=0$, for $1\le i$, ensures that 
a resolution of $\Ring/B\t_{\Ring/A} \Ring/C$ by free $\Ring/A$-modules may be obtained by forming the tensor product of a resolution of $\Ring/B$ with a resolution of $\Ring/C$; hence, the Hilbert series of 
these rings are related by the following identity: 
$$\HH_{(\Ring/B)\t_{\Ring/A} (\Ring/C)}(t)=\frac{\HH_{\Ring/B}(t) \HH_{\Ring/C}(t)}{\HH_{\Ring/A}(t)}.$$
 We conclude that 
 $$\dim (\Ring/B\t \Ring/C) =\dim \Ring/B + \dim \Ring/C -\dim \Ring/A.$$
The ring $\Ring/C$ is equal to $R$; the ring $\Ring/B\t \Ring/C$ is equal to $\Ring/(B+C)=R/I$; and
$\dim \Ring/A-\dim \Ring/B= \grade I$. It follows that
$$\dim R/I=\dim R-\grade I.$$
\end{proof}

\end{document}